\documentclass[12pt]{amsart}

\usepackage[margin=1in]{geometry}
\usepackage{amsmath,amssymb,amsthm, tikz, enumerate}
\usepackage{hyperref,color, xcolor,skull,textcomp}
\usepackage[normalem]{ulem}

\newtheorem{theorem}{Theorem}
\newtheorem{lemma}[theorem]{Lemma}
\newtheorem{corollary}[theorem]{Corollary}
\newtheorem{proposition}[theorem]{Proposition}
\theoremstyle{remark} \newtheorem*{remark}{Remark}
\numberwithin{theorem}{section}
\numberwithin{equation}{section}

\newcommand{\rk}{{\text {\rm rk}}}
\newcommand{\Gal}{{\text {\rm Gal}}}
\newcommand{\ord}{{\text {\rm ord}}}
\newcommand{\Disc}{{\text {\rm Disc}}}
\newcommand{\textmod}{{\text {\rm mod}}}
\newcommand{\Q}{\mathbb{Q}}
\newcommand{\Z}{\mathbb{Z}}
\newcommand{\C}{\mathbb{C}}
\newcommand{\calO}{\mathcal{O}}

\title{Rank growth of elliptic curves in nonabelian extensions}

\author{Robert J. Lemke Oliver}
\address{Department of Mathematics, Tufts University, 503 Boston Ave, Medford, MA 02155}
\email{robert.lemke\_oliver@tufts.edu}

\author{Frank Thorne}
\address{Department of Mathematics, University of South Carolina, 1523 Greene St, Columbia, SC 29201}
\email{thorne@math.sc.edu}

\begin{document}

\begin{abstract}
Given an elliptic curve $E/\mathbb{Q}$, it is a conjecture of Goldfeld that asymptotically half of its quadratic twists will have rank zero and half will have rank one.  Nevertheless, higher rank twists do occur: subject to the parity conjecture, Gouv\^ea and Mazur constructed $X^{1/2-\epsilon}$ twists by discriminants up to $X$ with rank at least two.  For any $d\geq 3$, we build on their work to consider twists by degree $d$ $S_d$-extensions of $\mathbb{Q}$ with discriminant up to $X$. We prove that there are at least $X^{c_d-\epsilon}$ such twists with {positive rank},
where $c_d$ is a positive constant that tends to $1/4$ as $d\to\infty$. {Moreover, subject to a suitable parity conjecture, we obtain the same result for twists with rank at least two.}
\end{abstract}

\maketitle

\section{Introduction and statement of results}

Let $E/\mathbb{Q}$ be an elliptic curve and let $E_D/\mathbb{Q}$ be its twist by the field $\mathbb{Q}(\sqrt{D})$.   Goldfeld \cite{Goldfeld1979} has conjectured that as $D$ ranges over fundamental discriminants, asymptotically 50\% of the twists $E_D/\mathbb{Q}$ will have rank zero and 50\% will have rank one.  Following the work of Gross and Zagier \cite{GrossZagier1986} and Kolyvagin \cite{Kolyvagin1988} on the Birch and Swinnerton-Dyer conjecture in the late 1980's, it became of critical importance to demonstrate the existence of a twist, satisfying some additional splitting conditions, with \emph{analytic} rank one.  This was first achieved independently by Bump, Friedberg, and Hoffstein \cite{BumpFriedbergHoffstein1990} and Murty and Murty \cite{MurtyMurty1991}.  Together, these results imply that if the analytic rank of an elliptic curve $E/\mathbb{Q}$ is at most one, then its algebraic rank is equal to its analytic rank.

In the wake of these results, it became natural to search for twists of rank two or greater.  By employing an explicit construction, the squarefree sieve, and the then recently proven cases of the Birch and Swinnerton-Dyer conjecture, Gouv\^ea and Mazur \cite{GouveaMazur1991} were able to produce $\gg X^{1/2-\epsilon}$ discriminants $D$ with $|D|\leq X$ for which the analytic rank of $E_D/\mathbb{Q}$ is at least two; under the parity conjecture, these twists also have algebraic rank at least two.  Unconditional results on twists with algebraic rank at least two were estbalished by Stewart and Top \cite{StewartTop1995}, though with a worse exponent.

Motivated by the program of Mazur and Rubin on Diophantine stability (see, e.g., \cite{MazurRubin2015}), we may cast the above results as being about the growth of the rational points $E(K)$ relative to $E(\mathbb{Q})$ in quadratic extensions $K/\mathbb{Q}$.  In this work, we are interested in the analogous problem when $K$ is a degree $d$ $S_d$-extension of $\mathbb{Q}$. Let 
\[
\mathcal{F}_d(X) := \{ K/\mathbb{Q} : [K:\mathbb{Q}] = d, \mathrm{Gal}(\widetilde{K}/\mathbb{Q}) \simeq S_d, |\mathrm{Disc}(K)| \leq X\}
\]
where $\mathrm{Disc}(K)$ denotes the absolute discriminant of the extension $K/\mathbb{Q}$ and $\widetilde{K}$ denotes its Galois closure. 
Our main theorem is the following analogue of Gouv\^ea and Mazur's work:

\begin{theorem}\label{thm:general}
Let $E/\mathbb{Q}$ be an elliptic curve and let $d \geq 2$.  There is a constant $c_d>0$ such that for each $\varepsilon = \pm 1$, the number of fields $K \in \mathcal{F}_d(X)$ for which $\mathrm{rk}(E(K))> \mathrm{rk}(E(\mathbb{Q}))$ and 
the root number
$w(E,\rho_K) = \varepsilon$ is $\gg X^{c_d-\epsilon}.$

We may take $c_d = 1/d$ for $d\leq 5$, $c_6 = 1/5$, $c_7=c_8=1/6$, and
\[
c_d = \frac{1}{4} - \frac{d^2+4d-2}{2d^2(d-1)}
\]
in general.  In particular, we may take $c_d > 0.16$ always, and $c_d > 1/4 - \epsilon$ as $d\to\infty$.
\end{theorem}

Here the {root number} $w(E,\rho_K) = \frac{w(E_K)}{w(E)}$ is related to the analytic ranks of $E/\mathbb{Q}$ and $E/K$ as follows.
Let $L(s, E)$ and $L(s, E_K)$ be the Hasse-Weil $L$-functions associated to $E/\mathbb{Q}$ and its base change to $K$.
Under the Birch and Swinnerton-Dyer Conjecture, the ranks $\mathrm{rk}(E(\mathbb{Q}))$ and $\mathrm{rk}(E(K))$ are equal to the analytic ranks of these $L$-functions. 
Therefore, $\mathrm{rk}(E(K) - \mathrm{rk}(E(\Q))$ is conjecturally equal to the order of vanishing of 
$\frac{L(s, E_K)}{L(s, E)}$ 
at the central point $s = 1/2$.

This quotient is an $L$-function in its own right, the {\itshape non-abelian twist} $L(s, E, \rho_K)$ of $E$ by the standard representation $\rho_K$ of $\mathrm{Gal}(\widetilde{K}/\mathbb{Q}) \simeq S_d$.
(See Section \ref{sec:twist-props}.) This $L$-function is conjectured, and is in some cases known, to be analytic and to satisfy a self-dual functional equation sending $s\mapsto 1-s$ with root 
number $w(E, \rho_K)$.  (For example, this holds whenever $L(s,\rho_K)$ satisfies the strong Artin conjecture.)
This root number thus controls 
the parity of 
$\ord_{s = 1/2} \frac{L(s, E_K)}{L(s, E)}$.
Under either the Birch and Swinnerton-Dyer conjecture or the parity conjecture, this is the same as the parity of $\mathrm{rk}(E(K)) - \mathrm{rk}(E(\mathbb{Q}))$, and we obtain the following.

\begin{corollary}
Assuming the parity conjecture, the number of $K \in \mathcal{F}_d(X)$ for which $\mathrm{rk}(E(K)) \geq 2 + \mathrm{rk}(E(\mathbb{Q}))$ is $\gg X^{c_d-\epsilon}$, with $c_d$ as in
Theorem \ref{thm:general}.
\end{corollary}

Using known progress toward the Birch and Swinnerton-Dyer conjecture, we also obtain the following unconditional result on analytic ranks in the case $d=3$.

\begin{theorem}\label{thm:analytic}
Assume that the elliptic curve $E/\mathbb{Q}$ has at least one odd prime of multiplicative reduction.  Then the number of $K \in \mathcal{F}_3(X)$ for which the analytic rank of $L(s,E,\rho_K)$ is at least $2$, is $\gg X^{1/3-\epsilon}$.
\end{theorem}

A curious feature of Theorem \ref{thm:general} is that the constant $c_d$ approaches $1/4$ from below.  One might therefore hope that there is some easy improvement to Theorem \ref{thm:general} that resolves this quirk.  In fact, the value of $c_d$ presented is not always optimal: the proof of Theorem \ref{thm:general} makes use of the Schmidt bound $\# \mathcal{F}_d(X) \ll X^{(d+2)/4}$, and this has been improved for large values of $d$.  However, the net effect of this is minor, and the following result is not obviously improved by any stricter assumption on $\#\mathcal{F}_d(X)$.

\begin{theorem}\label{thm:field-improvement}
Let $d \geq 7$.  If $\#\mathcal{F}_d(X) \ll X^{\frac{d-3}{4}+\frac{1}{2d}+\epsilon}$, then we may take
\[
c_d = \frac{1}{4} - \frac{1}{2d}.
\]
in Theorem \ref{thm:general}.  In particular, this is unconditional for $d\geq 16052$.
\end{theorem}

In fact, while our method in principle might have the ability to produce exponents $c_d$ slightly larger than $1/4$, we presently only see how to do so under rather heavy assumptions.

\begin{theorem}\label{thm:overkill}
Assume either that the $L$-functions $L(s,E_K)$ for $K \in \mathcal{F}_d(X)$ are automorphic and satisfy the generalized Riemann hypothesis and the Birch and Swinnerton-Dyer conjecture, or that the bound $\#\mathrm{Cl}(K(E[2]))[2] \ll D_K^\epsilon$ holds for all $K \in \mathcal{F}_d(X)$ and all $\epsilon>0$.  Then Theorem \ref{thm:general} holds with
\[
c_d = \frac{1}{4} + \frac{1}{2(d^2-d)}.
\]
\end{theorem}

We now comment on what we expect to be true.  It is a folklore conjecture, strengthened by Bhargava \cite{Bhargava2007}, that there is a positive constant $a_d$ such that $\#\mathcal{F}_d(X) \sim a_d X$.  Based on the minimality philosophy, since fields $K \in \mathcal{F}_d(X)$ admit no nontrivial subfields and the root numbers $w(E,\rho_K)$ assume both signs, it is reasonable to expect that a version of Goldfeld's conjecture should hold.  That is, that the number of $K \in \mathcal{F}_d(X)$ for which $\mathrm{rk}(E(K))=\mathrm{rk}(E(\mathbb{Q}))$ and the number for which $\mathrm{rk}(E(K))=1+\mathrm{rk}(E(\mathbb{Q}))$ should each be asymptotic to $\frac{1}{2}a_d X$.  Furthermore, a na\"ive heuristic based on quantization of Tate-Shafarevich groups and Tate's version of the Birch and Swinnerton-Dyer conjecture over number fields suggests that perhaps the number of $K$ for which $\mathrm{rk}(E(K))=2+\mathrm{rk}(E(\mathbb{Q}))$ should be $X^{3/4+o(1)}$.

Thus, Theorem \ref{thm:general} -- which, to the best of our knowledge, provides the first general bounds as $d\to\infty$ for the number of $K \in \mathcal{F}_d(X)$ for which $\mathrm{rk}(E(K))>\mathrm{rk}(E(\mathbb{Q}))$ -- is presumably very far from the truth.  However, it is only known at present for $d\geq 6$ that $\#\mathcal{F}_d(X) \gg X^{1/2+1/d}$ due to recent work of Bhargava, Shankar, and Wang \cite{BhargavaShankarWang}.  This result is the culmination of a natural line of thought (constructing fields via writing down polynomials), so producing a stronger lower bound for $\#\mathcal{F}_d(X)$ will require a substantial new idea.  In particular, since we are conjecturally accessing in Theorem \ref{thm:general} fields for which the rank increases by at least two, based on the above discussion, it is reasonable to expect that the best possible version of Theorem \ref{thm:general} available with current methods can do no better than $c_d = 1/4 + 1/d$.  
We therefore view Theorem \ref{thm:general} as nearly optimal, though it would surely be desirable to bridge the small gap between our results and this limit.  
It is not clear to us at this time how to do so.


Finally, we discuss briefly other results on the growth of the Mordell--Weil group in non-quadratic extensions $K/\mathbb{Q}$.  Most notably for our purposes, V. Dokchitser \cite{Dokchitser2005} analyzed the root numbers of $L(s,E_K)$ for general $K$ and obtained many corollaries about analytic ranks.  His work is a crucial ingredient in controlling the root numbers in Theorem \ref{thm:general}.  Quite recently, Fornea \cite{Fornea} has shown that for many curves $E/\mathbb{Q}$, the analytic rank of $E$ increases over a positive proportion of $K \in \mathcal{F}_5(X)$, though his work does not control the algebraic rank nor does it access twists for which the rank increases by two.  In the large rank direction, in earlier work, by a consideration of root numbers, Howe \cite{Howe1997} showed that in Galois $\mathrm{PGL}_2(\mathbf{Z}/p^n\mathbf{Z})$-extensions, the rank increases dramatically if $-N_E$ is a quadratic nonresidue modulo $p$, where $N_E$ denotes the conductor of the curve $E/\mathbb{Q}$.  However, this result is of a somewhat different flavor than Theorem \ref{thm:general}, as Howe is specifically exploiting the fact that such fields admit many nontrivial subfields.  (We recall again that a field $K \in \mathcal{F}_d(X)$ admits no such subfields.)  Lastly, in the complementary direction, Mazur and Rubin \cite{MazurRubin2015} show that for every prime power $\ell^n$, there are infinitely many cyclic degree $\ell^n$ extensions over which the Mordell--Weil group does not grow, and David, Fearnley, and Kisilevsky \cite{DavidFearnleyKisilevsky2007} have formulated conjectures about the frequency with which the rank increases over prime degree cyclic extensions.

\section{Organization of the paper and the strategy of the proof}

We begin by explaining the ideas that go into the proof of Theorem \ref{thm:general}. 

To construct points on $E$ over degree $d$ number fields, we construct points in parametrized families over degree $d$ extensions of certain function fields $\mathbb{Q}(\mathbf{t})$ where 
$\mathbf{t}=(t_1,\dots,t_r)$ for some $r$. For example, for $d = 3$ we find a Weierstrass model $E\colon y^2 = f(x)$ for which $P_f(x, t) := 
f(x) - (x + t)^2$ defines an $S_3$-extension of $\Q(t)$.
By Hilbert irreducibility, most specializations $t = t_0 \in \Q$ define $S_3$-extensions $K/\Q$, over which $E$ visibly gains a point. Lemma \ref{lem:fractional-points} then establishes that these
`new' points usually increase the rank.

\smallskip
After proving some preliminary lemmas in Section \ref{sec:galois-background}, we devote Section \ref{sec:galois} to constructing $S_d$-extensions of $\Q(t)$ along the lines discussed above for $d = 3$.
The strategy is to prove that the Galois groups of specializations contain various cycle types. We first use Newton polygons to exhibit `long' cycles. We then argue that, for a suitable Weierstrass model
of $E$, there exists a prime $p$ and a specialization $P_f(x, t_0)$ such that $p$ divides the discriminant $P_f(x, t)$ and $p^2$ does not. This proves that the Galois group of $P_f$ contains 
a transposition, and (after a bit of group theory) that it is therefore $S_d$.

\smallskip
We thus obtain $S_d$-extensions $K$ over which $E$ gains a point of infinite order. We must then bound the multiplicity with which a given field arises.  We present two ways of doing so.  The first method is via an analysis of the squarefree part of the discriminant of $K$ and is carried out in Section \ref{sec:small-degree}.  This requires the transcendence degree of the function field $\mathbb{Q}(\mathbf{t})$ to be quite small, and so is the more efficient of the two methods only for $d\leq 8$.  

The second method, presented in Section \ref{sec:large-degree}, is based on a slight improvement to a geometry-of-numbers argument due to  Ellenberg and Venkatesh \cite{EV} that was originally used to bound $\#\mathcal{F}_d(X)$ from below. We
adapt their construction to only count fields over which $E$ gains a point. 
This allows the transcendence degree of the field $\mathbb{Q}(\mathbf{t})$ to be large, but with some loss of control over the multiplicities. The added freedom gained by the number of parameters outweighs this small loss once $d\geq 9$.

\smallskip

Finally, it remains to control the root numbers $w(E,\rho_K)$.  We do so using work of V. Dokchitser \cite{Dokchitser2005}.  We review his work, along with other useful properties of the twist, in Section \ref{sec:twist-props}.  The net effect is that to show that both root numbers occur frequently it suffices to show that we construct many fields $K$ and $K^\prime$ that are ``$p$-adically close'' for each $p \mid N_E$ but for which the discriminants $D_K$ and $D_{K^\prime}$ have different signs.  Assembling all of this, we obtain Theorem \ref{thm:general}.  The proof of Theorem \ref{thm:analytic} relies on similar arguments from Section \ref{sec:small-degree} for small degrees, but it requires a slightly different handling of the root number.  This is provided to us by a different lemma of Dokchitser.

\section*{Acknowledgements}

The authors would like to thank Michael Filaseta, Jan Nekov\'a\v{r}, Jeremy Rouse, David Smyth, Stanley Yao Xiao, and David Zureick-Brown for useful insights on this problem.

This work was supported by NSF Grant DMS-1601398 (R.J.L.O.), by a NSA Young Investigator Grant (H98230-16-1-0051, F.T.), and
by a grant from the Simons Foundation (563234, F.T.).

\section{Properties of the twist} \label{sec:twist-props}

Let $E/\mathbb{Q}$ be an elliptic curve and let $K \in \mathcal{F}_d(X)$.  Formally, the non-abelian twist $L(s, E, \rho_K)$ may be defined by the relation
\begin{equation}\label{eq:l_factor}
L(s,E_K) = L(s,E) L(s,E,\rho_K).
\end{equation}
In Dokchitser \cite{Dokchitser2005}, $L(s, E, \rho_K)$ is given a more intrinsic definition that we now briefly recall. Let $\rho_K$ be the standard $d - 1$ dimensional
representation of $\Gal(\widetilde{K}/\Q) \simeq S_d$, which we also regard as a continuous representation of $\Gal(\overline{\Q}/\Q)$. The usual Artin formalism provides a factorization
\begin{equation}\label{eq:artin}
\zeta_K(s) = \zeta(s) L(s, \rho_K)
\end{equation}
of the Dedekind zeta function $\zeta_K(s)$, where $L(s, \rho_K)$ is the Artin $L$-function associated to $\rho_K$. 

Now, let $T_\ell(E)$ be the $\ell$-adic Tate module associated to $E$, and write
\[
H_\ell(E) = \textnormal{Hom}(T_\ell(E) \otimes \Q_\ell, \Q_\ell) \otimes_{\Q_\ell} \C,
\]
which is a $2$-dimensional $G_\Q$-module. Then the $L$-function $L(s, E)$ is defined, as usual,
in terms of the action of $G_\Q$ on $H_\ell(E)$; its twist $L$-function $L(s, E, \rho_K)$ is defined analogously
in terms of the representation on $H_\ell(E) \otimes \rho_K$. The formula \eqref{eq:l_factor} is then the exact analogue 
of \eqref{eq:artin}, and is similarly proved.

We may also regard $L(s,E,\rho_K)$ as the Rankin--Selberg $L$-function $L(s, E \times \rho_K)$.  The analytic properties of Rankin--Selberg products are known when the two $L$-functions are attached to cuspidal automorphic forms; for example, see Cogdell \cite{Cogdell} for a wonderful summary.  The modularity theorem establishes that $L(s,E)$ is attached to a classical modular form, and the strong Artin conjecture asserts that every $L(s,\rho_K)$ is attached to an automorphic form.  Thus, we expect that $L(s,E,\rho_K)$ is always entire, but this is at present wide open in general.

In the special case that $K/\mathbb{Q}$ is an $S_3$ cubic, the strong Artin conjecture is known for $L(s,\rho_K)$, whereby 
the $L$-function $L(s, E, \rho_K)$ is known to be holomorphic.
We may further connect this $L$-function to the Mordell--Weil group, as we now explain. 

Given a field $K \in \mathcal{F}_3(X)$, there is a unique quadratic subfield $F$ of the Galois closure $\widetilde{K}$ known as the quadratic resolvent of $K$.  If $\psi_K$ is the cubic ray class character of $F$ corresponding to the extension $\widetilde{K}/F$, then $L(s,\psi_K) = L(s,\rho_K)$.  Correspondingly, there is an equality of $L$-functions $L(s,E,\rho_K) = L(s,E_F,\psi_K)$.  As in \eqref{eq:l_factor}, it follows that 
$$
L(s,E_{\widetilde{K}}) = L(s,E_F) L(s,E_F,\psi_K) L(s,E_F,\overline{\psi}_K),
$$
where $\overline{\psi}_K$ is the character conjugate to $\psi_K$.  In fact, even though $\psi_K$ and $\overline{\psi}_K$ are distinct characters, their associated $L$-functions are the same.  Similarly, $L(s,E,\psi_K) = L(s,E,\overline{\psi}_K)$ as analytic functions, so we conclude in particular that 
\begin{align*}
\mathrm{ord}_{s=1/2} L(s,E_{\widetilde{K}}) - \mathrm{ord}_{s=1/2} L(s,E_F) 
	&= 2\cdot \mathrm{ord}_{s=1/2} L(s,E_F,\psi_K) \\
	&= 2\cdot \mathrm{ord}_{s=1/2} L(s,E,\rho_K).
\end{align*}
In other words, the analytic rank of $L(s,E,\rho_K)$ controls the growth of the analytic rank of $E$ in the extension $\widetilde{K}/F$.

There is an arithmetic manifestation of this story as well.  Viewing the Mordell--Weil group $E(\widetilde{K}) \otimes \mathbb{C}$ as a finite dimensional Galois representation and decomposing
it into isotypic components, a bit of Galois theory shows that the $\rho_K$-isotypic component $E(\widetilde{K})^{\rho_K}$ satisfies
\begin{align*}
\dim_\mathbb{C} E(\widetilde{K})^{\rho_K}
	&= \mathrm{rk}(E(\widetilde{K})) - \mathrm{rk}(E(F)) \\
	&= 2\cdot (\mathrm{rk}(E(K)) - \mathrm{rk}(E(\mathbb{Q}))).
\end{align*}
The first line follows because $E(F) \otimes \C$ is the direct sum of the remaining isotypic components; the second because,
for each element $\tau \in \Gal(\widetilde{K}/\Q)$ of order two, $\rho(\tau)$ has eigenvalues $1$ and $-1$. In particular, we see that
the growth of the rank of the Mordell--Weil group in the extension $\widetilde{K}/F$ is controlled by its growth in $K/\mathbb{Q}$.

Combining these two perspectives, the Birch and Swinnerton-Dyer conjecture predicts that the analytic rank of $L(s,E,\rho_K)$ controls the multiplicity of $\rho_K$ in the representation $E(\widetilde{K})\otimes \mathbb{C}$, and thereby the growth of the rank.  While this conjecture is certainly still wide open, it is known in the case that the analytic rank is $0$ and the field $F$ is imaginary:

\begin{theorem}[Nekov\'a\v{r} \cite{Nekovar}, Theorem A\textquotesingle]\label{thm:cubic-bsd}
With notation as above, suppose that $F$ is an imaginary quadratic field and that $E$ does not have CM by an order in $F$.  If $L(1/2,E_F,\psi_K) \neq 0$, then $\mathrm{rk}(E(\widetilde{K})) = \mathrm{rk}(E(F))$.  \end{theorem}
Here the $L$-function is again normalized so that $s = \frac{1}{2}$ is at the center of the critical strip.

From Theorem \ref{thm:cubic-bsd} and the above discussion, we obtain the following corollary.

\begin{corollary}\label{cor:cubic-ranks}
Let $E/\mathbb{Q}$ be an elliptic curve and let $K \in \mathcal{F}_3(X)$ have negative discriminant.  Suppose that $E$ does not have CM by an order in the quadratic resolvent of $K$.  If $w(E,\rho_K) = +1$ and $\mathrm{rk}(E(K)) \neq \mathrm{rk}(E(\mathbb{Q}))$, then the analytic rank of $L(s,E,\rho_K)$ is at least $2$.
\end{corollary}
\begin{proof}
Since $w(E,\rho_K) = +1$, the analytic rank of $L(s,E,\rho_K)$ must be even.  On the other hand, the requirement that $K$ have negative discriminant guarantees that the quadratic resolvent $F$ is an imaginary quadratic field.  Thus, since $\mathrm{rk}(E(K)) \neq \mathrm{rk}(E(\mathbb{Q}))$ and $L(s,E,\rho_K) = L(s,E_F,\psi_K)$, Theorem \ref{thm:cubic-bsd} precludes the possibility that $L(s,E,\rho_K) \neq 0$.  This implies that $L(s,E,\rho_K)$ must have rank at least $2$, as claimed.
\end{proof}

We now recall the work of Dokchitser \cite{Dokchitser2005} on the root numbers $w(E,\rho_K)$.  In many cases (see his Theorem 16, for example), he determined exactly the value of $w(E,\rho_K)$.  We require only the following properties, obtained as a consequence of \cite[Theorem 16]{Dokchitser2005} and its surrounding discussion.

\begin{lemma}[Dokchitser]
If $K \in \mathcal{F}_d(X)$, then there is a factorization
\[
w(E,\rho_K) = w(E)^{d-1} w_\infty(E,\rho_K) \prod_p w_p(E,\rho_K)
\]
such that:
\begin{enumerate}
\item $w_p(E,\rho_K) = 1$ if $E$ has good reduction at $p$; 
\item $w_\infty(E,\rho_K) = \mathrm{sgn}(\mathrm{Disc}(K))$, the sign of the discriminant of $K$; and
\item if $p \mid N_E$, then $w_p(E,\rho_K)$ depends only on $\rho_E\!\mid_{G_{\mathbb{Q}_p}}$ and $\rho_K\!\mid_{G_{\mathbb{Q}_p}}$, where $\rho_E$ is the Galois representation attached to $E$ and $G_{\mathbb{Q}_p} = \mathrm{Gal}(\overline{\mathbb{Q}_p}/\mathbb{Q}_p) \subseteq \mathrm{Gal}(\overline{\mathbb{Q}}/\mathbb{Q})$.
\end{enumerate}
\end{lemma}

From this, we derive the following important corollary.

\begin{corollary}\label{cor:sign}
Let $K$ and $K^\prime \in \mathcal{F}_d(X)$ be such that $K \otimes \mathbb{Q}_p \simeq K^\prime \otimes \mathbb{Q}_p$ for all $p \mid N_E$.  Suppose that $\mathrm{sgn}(\mathrm{Disc}(K)) = - \mathrm{sgn}(\mathrm{Disc}(K^\prime))$.  Then $w(E,\rho_K) = - w(E,\rho_{K^\prime})$.
\end{corollary}

In proving Theorem \ref{thm:analytic}, we will need a slightly different way to control the root number $w(E,\rho_K)$.  In particular, we have \cite[Corollary 2]{Dokchitser2005}:

\begin{lemma}\label{lem:sign-jacobi}
Suppose that the conductor $N_E$ of $E$ is relatively prime to the discriminant $\mathrm{Disc}(K)$ of $K \in \mathcal{F}_d(X)$.  Then
\[
w(E,\rho_K) = w(E)^{d-1} \mathrm{sgn}(\mathrm{Disc}(K)) \left(\frac{\mathrm{Disc}(K)}{N_E}\right), 
\]
where $(\frac{\cdot}{\cdot})$ denotes the Kronecker symbol.
\end{lemma}

We close this section by showing that for almost all $K \in \mathcal{F}_d(X)$, if $E(K) \neq E(\mathbb{Q})$, then $\mathrm{rk}(E(K)) > \mathrm{rk}(E(\mathbb{Q}))$.

\begin{lemma}
\label{lem:fractional-points}
Let $E/\mathbb{Q}$ be an elliptic curve.  There is a constant $C_{E,d}$, depending only on $E$ and $d$, such that
\[
\#\{K\in\mathcal{F}_d(X): E(K) \neq E(\mathbb{Q}) \text{ but } \mathrm{rk}(E(K))=\mathrm{rk}(E(\mathbb{Q}))\} \leq C_{E,d}.
\]
\end{lemma}
\begin{proof}
For each $K$ counted, there must exist some prime $\ell\geq 2$ and some point $P \in E(K) \setminus E(\Q)$ for which  $\ell P\in E(\mathbb{Q})$ but $mP \not \in E(\mathbb{Q})$ for any $m < \ell$.
Since any field in $\mathcal{F}_d(X)$ has no non-trivial subfields, we must have $\mathbb{Q}(P) = K$ and, recalling our notation for the Galois closure, $\widetilde{\mathbb{Q}(P)} = \widetilde{K}$.
Now, any conjugate of $P$ differs from $P$ by some point of order $\ell$ in $E(\overline{\mathbb{Q}})$, so there must be at least one point of order $\ell$ defined over $\widetilde{\mathbb{Q}(P)} = \widetilde{K}$.

By work of Merel \cite{Merel}, there is an absolute constant $T(d!)$ such that $|E(L)_\mathrm{tors}|\leq T(d!)$ for any field $L$ of degree $d!$.  We therefore have $\ell \leq T(d!)$.  For each such $\ell$ and point $P$ as above, the field $\mathbb{Q}(P)$ depends only on the class of $\ell P$ in $E(\mathbb{Q})/\ell E(\mathbb{Q})$ and possibly the choice of an $\ell$-torsion point in 
$E(\overline{\mathbb{Q}})$. Hence only finitely many such fields arise, and this yields the lemma.
\end{proof}

\section{Useful results from Galois theory} \label{sec:galois-background}

In this section, we recall several useful results from Galois theory and we prove a few preliminary lemmas that will be useful in what is to come.

We start off by recalling the Hilbert irreducibility theorem in the following context.  Let $f(\mathbf{t},x) \in \mathbb{Q}(\mathbf{t})[x]$ be an irreducible polynomial of degree $d$ over $\mathbb{Q}(\mathbf{t})$ where $\mathbf{t} = (t_1,\dots,t_k)$.  This defines an extension $K = \mathbb{Q}(t)[x]/f(\mathbf{t},x)$ which need not be Galois closed over $\mathbb{Q}(\mathbf{t})$.  Let $L$ be its Galois closure, which we take to be generated by the polynomial $g(\mathbf{t},x)$, and we write $G = \mathrm{Gal}(L/\mathbb{Q}(\mathbf{t}))$.  For any $\mathbf{t}_0 \in \mathbb{Q}^k$, we let $f_{\mathbf{t}_0}$, $g_{\mathbf{t}_0}$, $K_{\mathbf{t}_0}$, $L_{\mathbf{t}_0}$, and $G_{\mathbf{t}_0}$ denote the associated objects obtained under specialization.

\begin{theorem}[Hilbert irreducibility]\label{thm:hit}
With notation as above, suppose $\mathbf{t}_0$ is such that $g_{\mathbf{t}_0}$ is irreducible over $\mathbb{Q}$.  Then the permutation representations of $G$ and $G_{\mathbf{t}_0}$ acting on the roots of $f$ and $f_{\mathbf{t}_0}$ are isomorphic.

Moreover, the above hypothesis holds for a proportion $1-o_H(1)$ of $\mathbf{t}$ inside any rectangular region in $\mathbb{Z}^k$ whose shortest side has length $H$.
\end{theorem}

This is classical, and we take the last claim (i.e., that $g_{\mathbf{t}_0}$ is irreducible for almost all $\mathbf{t}_0$) as `well known'. However, 
we will make frequent use of the isomorphism of permutation representations, and this feature is less commonly stated.  Therefore, we provide a short proof of this fact.

\begin{proof}
Let $\alpha \in \overline{\mathbb{Q}(\mathbf{t})}$ be a root of $g(\mathbf{t},x)$, so that $L = \mathbb{Q}(\mathbf{t})(\alpha)$.  Similarly,
let $\beta  \in \overline{\mathbb{Q}}$ be a root of $g_{\mathbf{t}_0}$ with $L_{\mathbf{t}_0} = \mathbb{Q}(\beta)$.

Since $L$ is Galois closed over $\mathbb{Q}(\mathbf{t})$, each automorphism $\sigma \in G$ is determined by the unique polynomial
$P_\sigma(x) \in \mathbb{Q}(\mathbf{t})[x]$ for which $\deg(P_\sigma) < |G|$ and $\sigma(\alpha) = P_\sigma(\alpha)$. Writing
$P_{\sigma, \mathbf{t}_0}(x) \in \mathbb{Q}[x]$ for the polynomial obtained by specializing $\mathbf{t}$ to $\mathbf{t_0}$, we see at once
that the map $\widetilde{\sigma} \colon \beta \mapsto P_{\sigma, \mathbf{t}_0}(\beta)$ is an automorphism of $L_{\mathbf{t}_0}$.

The map $\sigma \mapsto \widetilde{\sigma}$ is thus a homomorphism from $G$ to $G_{\mathbf{t}_0}$. It is injective since $g_{\mathbf{t}_0}$ is
irreducible, forcing each of the $P_{\sigma, \mathbf{t}_0}(\beta)$ to be distinct. Since $|G| = \mathrm{deg}(g_{\mathbf{t}_0})$, the set $\{P_{\sigma, \mathbf{t}_0}(\beta)\}_{\sigma \in G}$ forms 
a complete set of conjugates of $\beta$.  Thus, the map $\sigma \mapsto \widetilde{\sigma}$ is
surjective and hence an isomorphism.

The roots of $f$ can be written in the form $h_i(\alpha)$, where $h_i$ ranges over a set of
$d$ polynomials in $\mathbb{Q}(\mathbf{t})[x]$, each of degree less than $|G|$.
By construction, if $h$ and $h'$ are any two such polynomials 
with $\sigma(h(\alpha)) = h'(\alpha)$, we must have
$\widetilde{\sigma}(h_{\mathbf{t}_0}(\beta)) = h'_{\mathbf{t}_0}(\beta)$. But the roots of $f_{\mathbf{t}_0}$ are exactly the
$h_{i, \mathbf{t}_0}(\beta)$, so that the action of $\sigma$ on the $h_i(\alpha)$ corresponds exactly to the action of $\widetilde{\sigma}$ on the 
$h_{i, \mathbf{t}_0}(\beta)$. This is our desired isomorphism of permutation representations. 
\end{proof}

We derive the following important corollary to Theorem \ref{thm:hit} that will enable us to populate the Galois groups $\mathrm{Gal}(f(\mathbf{t},x)/\mathbb{Q}(\mathbf{t}))$.

\begin{corollary}\label{cor:cycles}
Suppose $f(\mathbf{t},x)$ is irreducible over $\mathbb{Q}(\mathbf{t})$.  If the permutation representation $\mathrm{Gal}(f(\mathbf{t}_0,x) / \mathbb{Q})$ contains an element of a given cycle type for a positive proportion of $\mathbf{t}_0 \in \mathbb{Q}^k$ when ordered by height, then the permutation representation of $\mathrm{Gal}(f(\mathbf{t},x)/\mathbb{Q}(\mathbf{t}))$ must contain an element of the same cycle type.
\end{corollary}

Corollary \ref{cor:cycles} gives a means to show that the Galois group $\mathrm{Gal}(f(\mathbf{t},x)/\mathbb{Q}(\mathbf{t}))$ contains elements with many different cycle types.  The following lemma then enables us to show that in many cases, this suffices to guarantee that $\mathrm{Gal}(f(\mathbf{t},x)/\mathbb{Q}(\mathbf{t})) \simeq S_d$.

\begin{lemma}\label{lem:S_d-criterion}
Suppose that $G$ is a subgroup of $S_d$ such that:
\begin{itemize}
\item $G$ contains a $d$-cycle and a transposition; and,
\item Either $G$ contains a $(d - 1)$-cycle, or $d \geq 5$ is odd and $G$ contains a $(d - 2)$-cycle.
\end{itemize}
Then $G = S_d$.
\end{lemma}

\begin{proof} When $G$ contains a $(d - 1)$-cycle, we recall the proof from 
\cite[Lemma 8.26]{milne}. After renumbering, suppose that the $(d - 1)$-cycle is
$(1 \ 2 \ 3 \cdots d - 1)$. Since $G$ is transitive, it will 
contain a conjugate of the transposition of the form $(i \ d)$, for some $i < d$. Conjugating
by the $(d - 1)$-cycle and its powers, we see that $G$ will contain $(i \ d)$ for all $i < d$,
and these elements generate $S_n$.

\smallskip
Now, suppose instead that $d \geq 5$ is odd and $G$ contains a $(d - 2)$-cycle. If $G$ contains
a transposition $(i \ j)$, where the $(d - 2)$-cycle fixes $i$ but not $j$, then an argument similar
to that above establishes that $G$ contains the full symmetric group on $i$ and the elements permuted by the
$(d - 2)$-cycle.
So $G$ contains a $(d - 1)$-cycle and we are reduced to the first case. 

Finally, we prove that $G$ must contain such a transposition. By transitivity, $G$ will contain a transposition
$(i \ j)$ where the $(d - 2)$-cycle fixes at least one of $i$ and $j$. If it fixes exactly one of them, we're done. Otherwise, choose a suitable power $\sigma$ of the $d$-cycle
so that $\sigma(i) = j$; since $d$ is odd, we have $\sigma(j) = k$ for some $k \neq i, j$. Then $G$ contains $(j \ k)$, which is the desired transposition.
\end{proof}

We establish the existence of elements with various cycle types in a few different ways.  To obtain transpositions, we make use of the following well-known lemma.

\begin{lemma}\label{lem:trans-criterion}
Let $f(x) \in \mathbb{Z}[x]$ be irreducible and suppose that for some prime $p$ not dividing the leading coefficient of $f$, $p\mid\mid \mathrm{Disc}(f)$.  Then the natural permutation representation of $\mathrm{Gal}(f(x)/\mathbb{Q})$ contains a transposition.
\end{lemma}
\begin{proof}
Let $L_p$ be the splitting field of $f(x)$ over $\mathbb{Q}_p$.  The claim follows upon observing that $L_p$ is a ramified quadratic extension of an unramified extension of $\mathbb{Q}_p$, so that $\mathrm{Gal}(L_p/\mathbb{Q}_p)$ contains a transposition, and recalling the inclusion $\mathrm{Gal}(L_p/\mathbb{Q}_p) \hookrightarrow \mathrm{Gal}(f(x)/\mathbb{Q})$.
\end{proof}

To find long cycle types, we use of the theory of Newton polygons.  Given a rational polynomial $f(x) = a_d x^d + \dots + a_0$, its {\itshape $p$-adic Newton polygon} is defined to be the lower convex hull of the points $(i,v_p(a_i))$.
It is a union of finitely many line segments
whose slopes match the valuations of the roots of $f$ over $\overline{\mathbb{Q}_p}$, with multiplicities equal to their horizontal lengths. See
\cite[Ch. II.6]{Neukirch} for a good reference.

The Newton polygon controls much of the behavior of $\mathrm{Gal}(f(x)/\mathbb{Q}_p)$.  For our purposes, the following lemma suffices.

\begin{lemma}\label{lem:newton}
Suppose that the Newton polygon of $f(x)$ as described above contains a line segment of slope $m/n$ with $\mathrm{gcd}(m,n)=1$.  Assume that the length of this segment is $n$ and that the denominator of every other slope is coprime to $n$.  Then $\mathrm{Gal}(f(x)/\mathbb{Q})$ contains an $n$-cycle.
\end{lemma}

\begin{proof}
The hypotheses ensure that the roots of valuation $m/n$ form a set of Galois conjugates over $\mathbb{Q}_p$.  Thus, $f(x)$ admits a factorization $f(x) = f_0(x) f_1(x)$ over $\mathbb{Q}_p$, say, where the roots of $f_0(x)$ are the roots of valuation $m/n$.  Since the degree of $f_0(x)$ is $n$ by assumption, it must cut out a totally ramified extension of $\mathbb{Q}_p$.  The result now follows from the inclusions $\mathrm{Gal}(f_0(x)/\mathbb{Q}_p) \subseteq \mathrm{Gal}(f(x)/\mathbb{Q}_p) \subseteq \mathrm{Gal}(f(x)/\mathbb{Q})$.
\end{proof}

Finally, we recall some basic facts concerning polynomial resultants.  The {\itshape resultant} of two polynomials $f(x) = a_0 x^n + a_1x^{n-1} + \dots + a_n$ and $g(x) = b_0 x^m + \dots + b_m$ is given by
\begin{equation}\label{eq:result}
\mathrm{Res}(f,g) = a_0^m b_0^n \prod_{f(\alpha)=g(\beta)=0} (\alpha-\beta) \ = (-1)^{nm} b_0^n \prod_{g(\beta) = 0} f(\beta),
\end{equation}
where the products run over roots of $f$ and $g$, counted with multiplicity.  The key lemma here is the following. 

\begin{lemma}\label{lem:disc-res}
Let $F(x) = a_0 x^n + \dots + a_n$ be a polynomial.  Then
\[
\mathrm{Disc}(F) = \frac{(-1)^{n(n-1)/2}}{a_0} \mathrm{Res}(F,F^\prime) = (-1)^{n(n-1)/2} n^n a_0^{n-1} \prod_{\beta: F^\prime(\beta)=0} F(\beta).
\]
\end{lemma}
\begin{proof}
See, e.g., \cite[Proposition IV.8.5]{Lang} for the first equality, and the second follows from \eqref{eq:result}.
\end{proof}

\section{Analysis of Galois groups in a family} \label{sec:galois}

We are finally ready to discuss the family of polynomials we will use to construct points on elliptic curves over number fields.  Let $E$ be an elliptic curve given by a Weierstrass equation $y^2 = f(x)$. 
We define a polynomial $P(x, t) = P_f(x, t) \in \mathbb{Z}[x, t]$ by
\begin{equation}\label{eq:Pf_def}
P_f(x, t) = 
 \begin{cases}
 t^2 x^d - f(x), & \text{ $d$ even}, \\
 x^{d - 3} f(x) - t^2,  & \text{ $d$ odd, $d\geq 5$}, \\
 f(x)-(x +t)^2, & d=3.
 \end{cases}
\end{equation}
By construction, for each specialization $t = t_0 \in \mathbb{Q}$, each of
$(x, t_0x^{d/2})$, $(x, t_0 x^{\frac{3 - d}{2}})$, and $(x,x+t_0)$ is respectively a point on 
$E(K)$, where \[
K := \mathbb{Q}[x] / (P(x, t_0)).
\]
This construction is exactly what we will use for small degrees, and it is a specialization of the construction we will use for larger $d$.  In either case, we wish to argue that, for many choices of $t_0$, $K$ will indeed define an $S_d$-number field.  In view of the Hilbert irreducibility theorem, Theorem \ref{thm:hit}, the key result in this section is thus the following.

\begin{proposition}\label{prop:Sd}
Given $E$, there exists a Weierstrass model $y^2 = f(x)$ of $E$, integral except possibly at a single prime, for which 
$\mathbb{Q}(t)[x]/ (P(x, t))$ is a field extension of $\mathbb{Q}(t)$ of degree $d$
whose Galois closure has Galois group $S_d$ over $\Q(t)$.
\end{proposition}

The first step is to construct a Weierstrass model for $E$ with various properties to be exploited later.

\begin{lemma}\label{lem:exists_w}
Given an elliptic curve $E/\Q$, an integer $a$, a real number $\alpha$, and any positive $\epsilon > 0$, there exists 
a rational Weierstrass model $E \colon y^2 = f(x) = x^3 + Bx^2 + Cx + D$ and distinct primes $p_1, p_2, p_3 \nmid 6d(d - 3)N_E$ satisfying the following properties:
\begin{enumerate}[(i)]
\item The coefficients $B, C, D$ are all in $\Z[\frac{1}{p_1}]$.
\item We have $p_2 \mid \mid D$ and $p_2 \nmid C$.
\item We have $f(x) \equiv (x + a)^3 \pmod{p_3}$.
\item The polynomial $f(x)$ is `close to' $(x + \alpha)^3$ in the Euclidean metric; namely, we have 
\[
|B - 3\alpha| < \epsilon, \ \ |C - 3 \alpha^2| < \epsilon, \ \ |D - \alpha^3| < \epsilon.
\]
\end{enumerate}
\end{lemma}

\begin{proof}
We begin with (ii).
Starting
with an integral model $y^2 = g(x) := x^3 + ax + b$ for $E$,
upon substituting $x + r$ for $x$ we obtain a model
of the form
\begin{equation}\label{eq:nice_w_model}
y^2 = f_r(x) = x^3 + 3rx^2 + \big( 3r^2 + a \big) x + \big( r^3 + ar + b \big).
\end{equation}
By Chebotarev density, we may choose a prime $p_2 \nmid \Disc(g)$ and some $r \in \Z/p_2\Z$ for which $g(r) \equiv 0 \pmod {p_2}$ and $g'(r) = 3r^2 + a \not \equiv 0 \pmod {p_2}$.
Because $p_2 \nmid g'(r)$,
distinct lifts of $r \pmod {p_2^2}$ will yield distinct values of $g(r) \pmod{p_2^2}$, so we may choose a lift of $r$ to $\Z$ such that $f_r(x)$
satisfies (ii).

To also obtain (iii), let $p_3$ be any prime not dividing $6d(d-3)\Delta_E p_2$ and replace $f_r(x)$ with $\widetilde{f_r}(x) := p_3^6 f_r\big(\frac{x + a p_2^2 \overline{p_2}^2}{p_3^2}  \big)$, where $p_2 \overline{p_2} \equiv 1 \pmod{p_3^2}$.

Finally, let $p_1$ be any prime not dividing $6d(d-3)\Delta_E p_2p_3$.  Let $u \in \Z[\frac{1}{p_1}]$ be such that $p_2^2 p_3 \mid u$ and such that
 $|u^i - \alpha^i| < \frac{\epsilon}{4}$ for $i = 1, 2, 3$. Then, for a sufficiently large positive integer $k$,
$y^2 = p_1^{-6k} \widetilde{f_r}\big(p_1^{2k}(x + u)\big)$ is a Weierstrass model for $E$ satisfying all the stated properties.
\end{proof}

\begin{lemma}\label{lem:large-cycles}
Let $E/\mathbb{Q}$ be an elliptic curve with Weierstrass model in the form guaranteed by Lemma \ref{lem:exists_w}.  Then $P(x,t)$ is irreducible over $\mathbb{Q}(t)$.  Moreover, if $d$ is even, then the Galois group $\mathrm{Gal}(P(x,t))$ contains both a $d$-cycle and a $(d-1)$-cycle, while if $d$ is odd, it contains both a $d$-cycle and a $(d-2)$-cycle.
\end{lemma}
\begin{proof}

Arguing separately for $d$ even and odd, we make various substitutions $t = t_0$ in $P_f(x, t)$, and inspect the resulting
Newton polygons over $\mathbb{Q}_p$ with $p = p_2$ as in Lemma \ref{lem:exists_w}(ii). We will conclude that 
$P_f(x, t_0)$ is irreducible over $\Q_p$ (and hence over $\Q$), and we will exhibit various cycles in the Galois group of $\mathbb{Q}(t)[x]/ (P(x, t)))$ over $\Q(t)$ thereby using Corollary \ref{cor:cycles}.

{\em $d \geq 4$ even}:  We consider two specializations, namely $t=p^{-d/2}$ and $t=p^{-1}$, from which we obtain a $d$-cycle and a $(d-1)$-cycle, respectively, using Lemma \ref{lem:newton} and Corollary \ref{cor:cycles}.  We present these two $p$-adic Newton polygons in turn.

\begin{center}
\begin{tikzpicture}[scale = 1.7]
    \draw[very thin,color=gray] (0,-1.6) [xstep = 1, ystep = 0.2] grid (8.0,0.2);    
    \draw (0, 0) -- (8, 0);
    \draw (0, -1.6) -- (8, 0.2); 
    \filldraw (0,-1.6) circle (0.05) node[left] {$(d,-d)$};
    \filldraw (8,0.2) circle (0.05) node[right] {$(0,1)$};
\end{tikzpicture}
Newton polygon over $\mathbb{Q}_p$ with $t = p^{-d/2}$: a $d$-cycle.
\end{center}

\begin{center}
\begin{tikzpicture}[scale = 1.7]
    \draw[very thin,color=gray] (0,-.4) [xstep = 1, ystep = 0.2] grid (8.0,0.2);    
    \draw (0, 0) -- (8, 0);
    \draw (0, -.4) -- (7, 0); 
    \draw (7, 0) -- (8, 0.2);
    \filldraw (0,-0.4) circle (0.05) node[left] {$(d,-2)$};
    \filldraw (7,0) circle (0.05) node[below right] {$(1,0)$};
    \filldraw (8,0.2) circle (0.05) node[right] {$(0,1)$};
\end{tikzpicture}
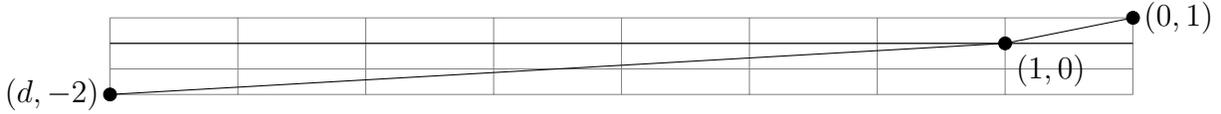
Newton polygon over $\mathbb{Q}_p$ with $t = p^{-1}$: a $(d - 1)$-cycle.
\end{center}

{\em $d = 3$}: Immediate.

{\em $d \geq 5$ odd}:  We take $t=p^{-1}$ and $t=p$, obtaining a $d$-cycle and a $(d-2)$-cycle, respectively, again using Lemma \ref{lem:newton} and Corollary \ref{cor:cycles}.
\begin{center}
\begin{tikzpicture}[scale = 1.7]
    \draw[very thin,color=gray] (0,-0.4) [xstep = 1, ystep = 0.2] grid (7.0,0);    
    \draw (0, 0) -- (7, 0);
    \draw (0, 0) -- (7, -0.4); 
    \filldraw (0,0) circle (0.05) node[left] {$(d,0)$};
    \filldraw (7,-0.4) circle (0.05) node[right] {$(0,-2)$};
\end{tikzpicture}

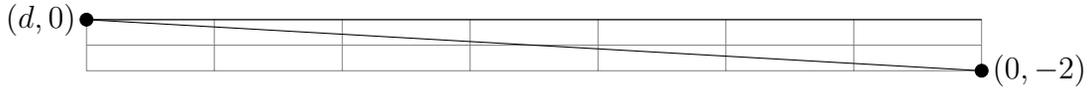
Newton polygon over $\mathbb{Q}_p$, with $t = p^{-1}$: a $d$-cycle. 
\end{center}

\begin{center}
\begin{tikzpicture}[scale = 1.7]
    \draw[very thin,color=gray] (0,0) [xstep = 1, ystep = 0.2] grid (7.0,0.4);    
    \draw (0, 0) -- (7, 0);
    \draw (2, 0) -- (7, 0.4); 
    \filldraw (0,0) circle (0.05) node[left] {$(d,0)$};
    \filldraw (2,0) circle (0.05) node[below] {$(d-2,0)$};
    \filldraw (7,0.4) circle (0.05) node[right] {$(0,2)$};
\end{tikzpicture}

Newton polygon over $\mathbb{Q}_p$, with $t = p$: a $(d - 2)$-cycle. 
\end{center}

This completes the proof.
\end{proof}

In view of Lemma \ref{lem:S_d-criterion}, to show that $\mathrm{Gal}(P(x,t)/\mathbb{Q}(t)) \simeq S_d$, it remains to show that the Galois group contains a transposition.  The key is the following computation.  We also recall from Corollary \ref{cor:sign} that to control the root numbers of these twists, we wish to control the sign of the discriminant of $P$.  We subsume the proof that we may do so into the following lemma.

\begin{lemma}\label{lem:disc-factor}
Given $E$, there exists a Weierstrass model of $E$ of the form given in Lemma \ref{lem:exists_w}, such that with $P_f(x, t)$ defined as in \eqref{eq:Pf_def},
the discriminant of $P_f$ (taken in the variable $x$)
is a non-squarefull polynomial in $t$ that assumes both positive and negative values in the interval $|t| \leq 1$.  This discriminant is of degree $4$ when $d = 3$ and is otherwise of the form
\[
\Disc(P_f) = t^{2d - 8} h(t)
\]
for a non-squarefull polynomial $h(t)$ of degree $6$.
\end{lemma}

\begin{proof}
We consider first the case that $d \geq 5$ is odd.  In this case, $P_f$ is monic and its discriminant is found via Lemma \ref{lem:disc-res} by taking the resultant of $P_f$ with its derivative $P_f^\prime$; namely, we have
\[
\mathrm{Disc}(P_f) 
	= (-1)^{(d-1)/2} d^d \prod_{\beta: P_f^\prime(\beta)=0} P_f(\beta)
\]
where the roots are taken with multiplicity.  For any Weierstrass model $y^2 = f(x)$ of $E$, we have $P_f^\prime = x^{d-4}[(d-3)f(x) + x f^\prime(x)] =: x^{d-4}g(x)$ for some cubic polynomial $g \in \mathbb{Q}[x]$.  Thus, $x=0$ is a root of $P_f^\prime$ with multiplicity $d-4$, and we conclude
\[
\mathrm{Disc}(P_f) = (-1)^{(d+1)/2}d^d t^{2d-8} \prod_{\beta: g(x) = 0} (\beta^{d-3}f(\beta) - t^2) = (-1)^{(d-1)/2} d^d t^{2d-8} h(t)
\]
for some monic degree $6$ polynomial $h \in \mathbb{Q}[t]$.  Choosing $f(x) \equiv (x + 1)^3 \pmod{p_3}$ in Lemma \ref{lem:exists_w}(iii), 
we have $P_f \equiv x^{d-3}(x+1)^3 - t^2 \pmod{p_3}$ and $\Disc(P_f) \equiv \Disc(x^{d-3}(x+1)^3 - t^2) \pmod{p_3}$.   By an argument with resultants similar to the above, we find
\begin{equation}\label{eqn:disc-1}
\mathrm{Disc}(x^{d-3}(x+1)^3 - t^2) = (-1)^{(d-1)/2} t^{2d-4} (d^d t^2 - 27 (d-3)^{d-3}),
\end{equation}
which is not squarefull when reduced $\pmod{p_3}$.  Thus, $\mathrm{Disc}(P_f)$ cannot be squarefull.  

To ensure that $\mathrm{Disc}(P_f)$ assumes both positive and negative values in the interval $|t|\leq 1$, 
choose $f$ close to $(x + 1)^3$ in the Euclidean topology, by Lemma \ref{lem:exists_w}(iv).
 As $\mathrm{Disc}(x^{d-3}(x+1)^3-t^2)$ visibly has the desired property thanks to \eqref{eqn:disc-1}, so does $P_f(x,t)$ by continuity.
 
 \smallskip
 
In the case that $d\geq 4$ is even, we exploit the fact that the discriminant of a polynomial and its reciprocal polynomial are the same, i.e. $\mathrm{Disc}(P_f(x)) = \mathrm{Disc}(x^d P_f(1/x))$.  The polynomial $x^d P_f(1/x)$ is of essentially the same form as the polynomials $P_f(x)$ for $d$ odd, and exactly the same argument shows that $\mathrm{Disc}(x^d P_f(1/x)) = t^{2d-8} h(t)$ for some sextic polynomial $h$.

To show that $\mathrm{Disc}(P_f)$ is not squarefull, choose $f(x) \equiv (x-1)^3 \pmod{p_3}$. As $P_f \equiv t^2 x^d - (x-1)^3 \pmod{p_3}$ and
\[
\mathrm{Disc}(t^2x^d - (x-1)^3) = \mathrm{Disc}(x^{d-3}(x-1)^3 + t^2)  = (-1)^{d/2}t^{2d-4}(d^dt^2 -27 (d-3)^{d-3}),
\]
it follows as in the odd case that $\mathrm{Disc}(P_f)$ is not squarefull. Similarly, by choosing $f$ close to $(x - 1)^3$ in the Euclidean topology, we ensure that $\mathrm{Disc}(P_f)$ assumes both positive and negative values in the interval $|t|\leq 1$.

Finally, if $d=3$, $P_f(x,t) = f(x) - (x+t)^2$ and $\mathrm{Disc}(P_f)=h(t)$ is a degree four polynomial in $t$.  
Choose a Weierstrass model for $f$ close, in $\mathbb{R}$, to $y^2 = x^3$; since  $\mathrm{Disc}(x^3-(x+t)^2) = -t^3(27t+4)$, $h(t)$ will assume positive and negative values inside $|t|\leq 1$.
Since a squarefull degree polynomial of degree four is either a square or a fourth power, this also proves that $h(t)$ is not squarefull.
\end{proof}

We are now ready to argue that the Galois group of $K$ contains a transposition.

\begin{lemma}\label{lem:exists-trans}
Let $E/\mathbb{Q}$ be an elliptic curve with Weierstrass model given by Lemma \ref{lem:disc-factor}.  Then $\mathrm{Gal}(P_f(x,t)/\mathbb{Q}(t))$ contains a transposition in its natural permutation representation.
\end{lemma}
\begin{proof}
As expected, we use Lemma \ref{lem:trans-criterion}.  If $E$ is given by a Weierstrass model of the form given by Lemma \ref{lem:disc-factor}, then $P_f(x,t)$ is irreducible and $\mathrm{Disc}(P_f) = t^{2d-8} h(t)$ for some non-squarefull polynomial $h(t) \in \mathbb{Z}[t]$ of degree $6$, or degree $4$ in the special case $d=3$.  Since $h(t)$ is not squarefull, it admits an irreducible factor $h_0(t)$ of multiplicity one.  Moreover, the proof of Lemma \ref{lem:disc-factor} shows that we may take $h_0(t) \neq t$.  If we write $h(t) = h_0(t) h_1(t)$, then only finitely many primes divide the resultant $\mathrm{Res}(h_0(t), th_1(t))$.  By the Chebotarev density theorem, there are infinitely many primes $p$ for which $h_0(t)$ admits a root.  Let $p$ be such a prime for which $p\nmid \mathrm{Disc}(h_0(t))$ and $p\nmid \mathrm{Res}(h_0(t),th_1(t))$.  By the definition of the resultant, we may thus find an integer $t_0$ for which $p\mid\mid h_0(t_0)$ and $p\nmid t_0 h_1(t_0)$.  Thus, $p\mid\mid \mathrm{Disc}(P_f(x,t_0))$ and $\mathrm{Gal}(P_f(x,t_0))$ contains a transposition by Lemma \ref{lem:trans-criterion}.  In particular, this construction shows that $\mathrm{Gal}(P_f(x,t_0)/\mathbb{Q})$ has a transposition for a positive proportion of $t_0 \in \mathbb{Q}$, which by Corollary \ref{cor:cycles} implies that $\mathrm{Gal}(P_f(x,t)/\mathbb{Q}(t))$ must also contain a transposition.
\end{proof}

Combining Lemmas \ref{lem:large-cycles} and \ref{lem:exists-trans} with Lemma \ref{lem:S_d-criterion}, we conclude Proposition \ref{prop:Sd}.

\section{Disambiguation via discriminants and small degree fields}\label{sec:small-degree}

The main point of this section is to establish the following theorem, which forms part of our main theorem.  At the end of this section, 
we then tweak the proof to obtain a proof of Theorem \ref{thm:analytic}.

\begin{theorem}
\label{thm:small-degree}
Let $E/\mathbb{Q}$ be an elliptic curve let $d\geq 3$ be an integer.  There is a constant $c_d>0$ such that for each $\varepsilon =\pm 1$, there are $\gg X^{c_d - \epsilon}$ fields $K\in\mathcal{F}_d(X)$ with $w(E,\rho_K)=\varepsilon$ and $\mathrm{rk}(E(K)) > \mathrm{rk}(E(\mathbb{Q}))$.  In particular, we may take
\[
c_d = \left\{ \begin{array}{ll} 1/3, & \text{if } d=3, \\ 1/4, & \text{if } d=4, \text{ and} \\ (\lceil \frac{d}{2}\rceil +2)^{-1}, & \text{if } d\geq 5. \end{array} \right.
\]
\end{theorem}

Recall that Proposition \ref{prop:Sd} yielded a Weierstrass model $y^2 = f(x)$ of $E$ and a polynomial $P_f(x, t)$ of \eqref{eq:Pf_def} defining an $S_d$-extension of $\Q(t)$, such that
each specialization $t = t_0 \in \mathbb{Q}$ yields a point on $E(K)$ with $K := \mathbb{Q}[x]/(P_f(x, t_0))$. 

We will choose specializations $t_0 = u/v$ where $u$ and $v$ range over integers in a suitably sized box.
The next two lemmas, applied to a homogenization of the polynomial $h(t)$ from Lemma \ref{lem:disc-factor},
will be used to show that the discriminants of the $P_f(x, u/v)$, as polynomials in $x$, represent many different square classes in $\Q^{\times} / (\Q^{\times})^2$ --
and hence that these polynomials generate many different field extensions.

\begin{lemma}[Greaves]
\label{lem:greaves}
Let $F(u,v)$ be an integral binary form with each irreducible factor of degree $\leq 6$.  Let $M\geq 1$ be a fixed positive integer and let classes $a,b\pmod{M}$ be chosen so that $F(u,v)$ does not admit a constant square factor whenever $u\equiv a\pmod{M}$ and $v \equiv b \pmod{M}$. Let $\Omega \subset [-1,1]^2$ be a smooth domain with volume $\mathrm{vol}(\Omega)$ and for any $U>1$, let $U\cdot\Omega$ denote the dilation of $\Omega$ by $U$.  Then there is a positive constant $c_F$, depending on $M$ but independent of $\Omega$, for which
\begin{equation}\label{eq:greaves}
\#\{u,v \in U\cdot \Omega : (u, v) \equiv (a, b) \ (\textmod \ M), F(u,v) \text{ squarefree}\} = c_F \mathrm{vol}(\Omega) U^2 + O\left(\frac{U^2}{(\log U)^{1/3}}\right).
\end{equation}
\end{lemma}
\begin{proof}
This is essentially the main theorem of \cite{Greaves1992}, which is stated in the slightly simpler case $\Omega = (0 ,1]^2$. The result is easily extended to 
$\Omega = [-1, 1]^2$ by considering $F(\pm u, \pm v)$. Greaves's proof is then easily modified as follows: 

Writing $N(U)$
for the quantity in \eqref{eq:greaves}, Greaves
writes
\[
N(U) = N'(U) + O(E(U)),
\]
where the `principal term' $N'(U)$ counts those $(u, v)$ such that $F(u, v)$ has no square factor $p^2$ with $p \leq \frac{1}{3} \log(x)$,
and where the `tail estimate' $E(U)$ is an error term. 

The quantity $N'(U)$ is easily estimated using inclusion-exclusion and the geometry of numbers,
and these methods extend immediately when $[-1, 1]^2$ is replaced with a more general $\Omega$. Meanwhile, the tail estimate for $\Omega$ is bounded by that for
$[-1, 1]^2$, and thus the error term may be quoted from \cite{Greaves1992} without change.
\end{proof}

 \begin{remark}
With a further generalization of Lemma \ref{lem:greaves} to skew boxes, we could improve our main result for small $d$.
For example, when $d = 3$, we have $\Disc(K) \mid v^2 H(u, v)$ for a quartic form $H$, and we would improve our results if we could replace 
$U \cdot \Omega$ with a region
approximating $[-X^{1/4}, X^{1/4}] \times [- X^{1/6}, X^{1/6}]$.
 \end{remark}

\begin{lemma} \label{lem:multiplicity}
Let $F(u,v)$ be a homogeneous rational binary form of degree $m$, and let $U,V\geq 1$.  For any integer $n$, there are $O_F(U^\epsilon V^\epsilon |n|^\epsilon)$ integral solutions to the equation $F(u,v)=n$ with $|u|\leq U$ and $|v|\leq V$.
\end{lemma}
\begin{proof}
We may choose a fixed finite extension $L/\Q$ and factorization
\[
F(u,v) = \frac{1}{k} \prod_{i=1}^m (\alpha_i u + \beta_i v),
\]
for some integer $k$ and algebraic integers $\alpha_i, \beta_i \in \mathcal{O}_L$.  Observe that if $u,v\in\mathbb{Z}$, then $|\alpha_i u + \beta_i v|_\nu \ll U+V$ for each infinite place $\nu$ of $L$.  

Each solution to $F(u,v)=n$ determines a factorization $nk\mathcal{O}_L = \mathfrak{a}_1 \dots \mathfrak{a}_m$ into principal ideals $\mathfrak{a}_i$ of $\mathcal{O}_L$, and
there are $O(n^\epsilon)$ such factorizations. Moreover, writing 
$r$ for the unit rank of $L$, there are at most $O(\log(U+V)^r)$ generators $\gamma_i = \alpha_i u + \beta_i v$ of each ideal $\mathfrak{a}_i$ for which $|\gamma_i|_\nu \ll U+V$ for each infinite place $\nu$. 
The result follows. 
\end{proof}

We are now ready to prove the main theorem of this section.

\begin{proof}[Proof of Theorem \ref{thm:small-degree}]
Let $E$ be given by the Weierstrass model produced in Proposition \ref{prop:Sd}, so that the polynomial $P_f(x,t)$ defined in \eqref{eq:Pf_def} cuts out an $S_d$ extension of $\mathbb{Q}(t)$.  
The polynomial $v^2 P_f(x,u/v)$  has coefficients integral away from a single fixed prime, and by Lemma \ref{lem:disc-factor}, it has discriminant of the form $u^{2d-8}v^{2d-2} H(u,v)$
 for some binary sextic form $H(u,v)$ that is not squarefull. For $d = 3$, the discriminant is instead of the form $v^4 H(u, v)$ with $H$ quartic instead of sextic.

By Hilbert irreducibility (Theorem \ref{thm:hit}), for asymptotically 100\% of pairs $(u, v)$ with $|u|, |v| \leq U$, we will have
that $K=\Q[x]/(P_f(x,u/v))$ is an $S_d$-field extension of $\Q$. We have $v_p(\Disc(K)) \leq p - 1$ for any tamely ramified prime $p$, 
 and $\Disc(K)$ and $\Disc(v^d P_f(x,u/v))$ differ by a rational square. Therefore, $\Disc(K)$ divides a bounded factor times either $u^{d-2}v^{d-2} H(u,v)$ or $u^{d-1}v^{d-1} H(u,v)$, depending on whether $d$ is even or odd.  
 Thus, there is some constant $q_{E,d} > 0$ such that taking $U = q_{E,d} X^{c_d/2}$ guarantees that $|D_K| \leq X$. Finally, Lemmas \ref{lem:greaves} and \ref{lem:multiplicity} guarantee that 
 $H(u, v)$, and hence $\Disc(K)$, represents $\gg X^{c_d - \epsilon}$ distinct square classes, so that  $\gg X^{c_d - \epsilon}$ distinct fields $K$ are produced.
 
 \smallskip
By Lemma \ref{lem:fractional-points}, we have $\rk(E(K)) > \rk(E(\Q))$ for all but a bounded number of these $K$. It remains to control the sign of the root number.  
Lemma \ref{lem:disc-factor} shows that both regions $\Omega^\pm:=\{(u,v) \in [-1,1]^2 : \pm \mathrm{Disc}(P_f(x,u/v))>0\}$ have positive volume. 
By Corollary \ref{cor:sign}, there exists a residue class $(u_0, v_0) \pmod{M}$ (with $M$ a suitably large power of $N_E$), for which 
$w(E,\rho_{K_0})$ is determined by the sign of $\mathrm{Disc}(P_f(x,u/v))$ whenever $(u, v) \equiv (u_0, v_0) \pmod{M}$.  
We incorporate the conditions that $(u, v) \in \Omega^\pm$ and that $(u, v) \equiv (u_0, v_0) \pmod{M}$ into our application of 
Lemma \ref{lem:greaves}, and the remainder of our proof is unchanged.
\end{proof}

Using very similar ideas, we prove Theorem \ref{thm:analytic} on non-abelian cubic twists with analytic rank two.

\begin{proof}[Proof of Theorem \ref{thm:analytic}]
The proof follows that of 
Theorem \ref{thm:small-degree}, except that to apply Corollary \ref{cor:cubic-ranks} we must produce {complex} cubic fields $K$ for which $w(E, \rho_K) = +1$. Accordingly, 
we use Lemma \ref{lem:sign-jacobi} instead of Corollary \ref{cor:sign} to control the root number $w(E,\rho_K)$.  
In the event that $E$ has CM, there is one exceptional quadratic resolvent for which we may not apply Corollary \ref{cor:cubic-ranks}.  However, the quadratic resolvent of $K \in \mathcal{F}_3(X)$ is determined by the squarefree part of its discriminant.  We distinguish fields in the above proof precisely by the squarefree part of their discriminant, so this one possible exceptional field has no impact on the result.

In Lemma \ref{lem:exists_w}, after (ii) but before the remaining steps, we replace $f(x)$ with $N_E^6 f(x N_E^{-2})$, allowing
us to demand that 
$f(x) \equiv x^3 \pmod{N_E}$, so that
\[
\mathrm{Disc}(P_f(x,t)) \equiv \mathrm{Disc}(x^3-(x+t)^2) \equiv -t^3(27t+4) \pmod{N_E}.
\]
For each odd prime $p$ for which $p \mid \mid N_E$, 
an easy argument shows that the polynomial $27t^2+4t$ represents both squares and nonsquares $\pmod{p}$.
Since by hypothesis there is at least one such prime, suitable congruence conditions on $t \pmod{N_E}$ may be chosen to guarantee that both $\mathrm{gcd}(\mathrm{Disc}(P_f(x,t)),N_E) = 1$ and $\left(\frac{\mathrm{Disc}(P_f(x,t))}{N_E}\right) = -1$.  
The result now follows as in the proof of Theorem \ref{thm:small-degree}.
\end{proof}

\section{Geometry of numbers and large degree fields} \label{sec:large-degree}

 In this section we prove the following complement to Theorem \ref{thm:small-degree}:
 
\begin{theorem}
\label{thm:large-degree}
Let $E/\mathbb{Q}$ be an elliptic curve let $d\geq 5$ be an integer.  Then, for each $\varepsilon =\pm 1$, there are $\gg X^{c_d - \epsilon}$ fields $K\in\mathcal{F}_d(X)$ with $w(E,\rho_K)=\varepsilon$ and $\mathrm{rk}(E(K)) > \mathrm{rk}(E(\mathbb{Q}))$, with
\[
c_d = \frac{1}{4} - \frac{d^2+4d-2}{2d^2(d-1)}.
\]
If $d \geq 16052$, then we may take
\[
c_d = \frac{1}{4} - \frac{1}{2d}.
\]
\end{theorem}

The result is identical to Theorem \ref{thm:small-degree} except for the value of $c_d$.
Here it is an increasing function of $d$, and this result improves upon Theorem \ref{thm:small-degree} for $d \geq 9$.

\smallskip

Our strategy is to adapt Ellenberg and Venkatesh's proof of a lower bound \cite{EV} for $\#\mathcal{F}_d(X)$. They produce
many algebraic integers $\alpha$ for which $|\Disc(\Z[\alpha])| < X$, and then, for each field $K$, bound from above the number of $\alpha$ so constructed with
$\Q(\alpha) = K$. We adapt their construction so as to produce only those $\alpha$ for which 
there are polynomials $F(x),G(x) \in \mathbb{Z}[x]$ such that $\big(\alpha,\frac{F(\alpha)}{G(\alpha)}\big)$ is a point on $E(\overline{\mathbb{Q}})$.  Equivalently, if $E$ is given by the Weierstrass model $E\colon y^2=f(x)$, we only count those $\alpha$ arising as solutions to $F(x)^2 - f(x)G(x)^2 = 0$ for some $F$ and $G$.


\subsection{The construction} 
Let $d \geq 4$. 
Using Lemma \ref{lem:exists_w} to choose a Weierstrass model for $E/\mathbb{Q}$, we consider the following family of polynomials.
Fix a parameter $Y$ to be chosen shortly. The construction
is slightly different depending on whether $d$ is odd or even.

If $d$ is even, we take:
\begin{itemize}
\item
$F(x) = x^{\frac{d}{2}} + a_1 x^{\frac{d}{2} - 1} +  a_2 x^{\frac{d}{2} - 2} + \dots + a_{d/2}$, an integral monic polynomial of degree $\frac{d}{2}$ with 
and $|a_k| \leq Y^k$ for each $k$.
\item
$G(x) = b_2 x^{\frac{d}{2} - 2} + b_3 x^{\frac{d}{2} - 3} + \dots + b_{d/2}$, an integral polynomial of degree $\frac{d}{2} - 2$ with
$|b_k| \leq Y^{k - \frac{3}{2}}$ for each $k$.
\end{itemize}

If $d$ is odd, we instead take: 
\begin{itemize}
\item
$G(x) = x^{\frac{d - 3}{2}} + b_1 x^{\frac{d - 3}{2} - 1} +  b_2 x^{\frac{d - 3}{2} - 2} 
+ \dots + 
b_{\frac{d - 3}{2}}$,
with 
$|b_k| \leq Y^k$ for each $k$.
\item
$F(x) = a_0 x^{\frac{d - 1}{2}} + a_1 x^{\frac{d - 1}{2} - 1}
+ \dots + 
a_{\frac{d - 1}{2}}$, with
$|a_k| \leq Y^{k + \frac{1}{2}}$ for each $k$.
\end{itemize}

\medskip
In either case, the polynomial
\begin{equation}\label{eq:def_H}
H(x) := F^2 - f G^2 = x^d + c_1 x^{d - 1} + c_2 x^{d - 2} + \cdots + c_d
\end{equation}
has $|c_k| \ll_{f, d} Y^k$ for each $k$, so that
$|\Disc(H)| \ll_{f, d} Y^{d(d -1)}$.  Thus, we will ultimately take $Y=q_{f,d} X^{1/d(d-1)}$ for a suitable constant $q_{f,d}$.
In general $H$ is not required to have integral coefficients (because $f$ isn't), but $H$ will have rational
coefficients whose denominators are bounded above by a fixed constant (depending on $E$ and $d$).  

\begin{lemma}
Let $R$ be the polynomial ring obtained by adjoining all the $a_i$ and $b_j$ as indeterminates to 
$\Z[\frac{1}{p_1}]$.

Then, as a polynomial in $R[x]$, $H$ is irreducible with Galois group $S_d$.
\end{lemma}

\begin{proof}
It suffices to exhibit specializations of the $a_i$ and $b_j$ to the polynomials described in 
\eqref{eq:Pf_def}, proved to be irreducible over $\Q(t)$ with Galois group $S_d$.

When $d$ is odd, choose $F = t$ and $G = x^{\frac{d - 3}{2}}$. This yields $H = - (x^{d - 3} f(x) - t^2)$, which is the same as \eqref{eq:Pf_def} up to a sign.

When $d$ is even, choose $F = x^{d/2}$ and $G = t$, obtaining
$H(x, t) = x^d - t^2 f(x)$. The polynomial $t^2 H(x, t^{-1}) = t^2 x^d - f(x)$ also appeared in \eqref{eq:Pf_def} and was previously proved irreducible over $\Q(t)$ with Galois group $S_d$.
Since the map $t \rightarrow t^{-1}$ induces an automorphism of $\Q(t)$, the same is true of $H(x, t)$.
\end{proof}

The following lemma establishes that we can control the discriminant of $H$, thereby allowing us to use 
Corollary \ref{cor:sign} to control the root number $w(E, \rho_K)$.

\begin{lemma}\label{lem:control_sign}
Suppose we are given a fixed choice of $H_0(x)$ as in \eqref{eq:def_H}, a positive integer $M$ coprime to the
denominators of the coefficients of $f$, and a choice of sign $\delta \in \pm 1$.

Then, as $Y \rightarrow \infty$, a positive proportion of the polynomials $H$ constructed above satisfy $H \equiv H_0 \pmod{M}$
and $\textnormal{sgn}(\Disc(H)) = \delta$.
\end{lemma}

\begin{proof}
We will exhibit choices of $F$ and $G$ with the $a_i$ and $b_j$ real numbers in $(-1, 1)$ for which 
$\Disc(F^2 - x^3 G^2)$ is positive and for which it is negative. 

Once this is done, the lemma quickly follows: for each $H$, define $H_Y(x) = Y^{-d} H(xY)$; equivalently, 
divide each $c_i$ in \eqref{eq:def_H} by $Y^i$. Then 
$\textnormal{sgn}(\Disc(H)) = \textnormal{sgn}(\Disc(H_Y))$. 
Since $Y^{-3} f(xY)$ tends to $x^3$ as $Y \rightarrow \infty$, and since the discriminant of a polynomial
is a continuous function of the coefficients, a positive proportion of the $H$ constructed will satisfy $H \equiv H_0 \pmod{M}$
and will have $H_Y$ sufficiently close to $F^2 - x^3 G^2$ as to guarantee that their discriminants are of the same sign.

Our $F$ and $G$ are chosen in an ad hoc manner. When $d$ is even,
choose
\[
F(x) = \Big( x^{d/2} + \frac{1}{100} \Big), \ \ G(x) = \lambda,
\]
and set $T(x) := F(x)^2 - x^3G(x)^2$.  We now recall Descartes's \emph{rule of signs}, that the number of positive roots of a real polynomial is bounded by the number of sign changes in its consecutive non-zero coefficients.
When $\lambda = \frac{1}{100}$, $T(x)$ is always positive and has no real roots.  When $\lambda = \frac{9}{10}$,
$T(x)$ has exactly two real roots by Descartes's rule of signs and because $T(\frac{1}{2})$ is negative.
Therefore, these two choices of $\lambda$ lead to opposite signs for $\Disc(T)$.

Similarly, when $d$ is odd, choose
\[
F(x) = (x + \lambda)^2, \ \ G(x) = x^{\frac{d - 3}{2}},
\]
and again set $T(x) = F(x)^2-x^3G(x)$.  Then $T(x)$ has an odd number of real roots. When $\lambda = \frac{1}{10}$, $T$ has exactly one real root by Descartes's rule.  When
$\lambda = - \frac{1}{10}$, $T$ may have either one or three roots.  We have $T(0) > 0$, $T(\frac{1}{10}) < 0$, and $T(\frac{1}{5}) > 0$, so that $T(x)$ has three real roots in this case.
We once again obtain opposite signs for $\Disc(T)$.
\end{proof}

\subsection{Bounding multiplicities}
There are two sources of multiplicity with which a single field $K$ can arise from multiple choices of the
$a_i$ and $b_j$.  We first bound the number of times in which a given polynomial $H$ can occur in
the construction \eqref{eq:def_H}.

\begin{lemma}\label{lem:bound_mult}
Let $H(x)$ and $f(x)$ be polynomials in $\Z[\frac{1}{p_1}][x]$ of degree $d$ and $3$ respectively.  Then
the number of polynomials $F(x), G(x) \in \Z[x]$ with $F^2 - fG^2 = H$ and with at least one of $F$ and $G$ monic, is $O_d(1)$.
\end{lemma}

\begin{proof}
To each way of writing $H = F^2 - fG^2$ we associate the factorization $H = (F - G \sqrt{f})(F + G \sqrt{f})$
in the coordinate ring $\C[x][\sqrt{f}] = \C[x, y]/(y^2 - f)$ of our elliptic curve.
This ring is a Dedekind domain \cite[Theorem II.5.10]{lorenzini}, so the ideal $(H)$ factors uniquely 
as a product of prime ideals, each of the form $(x - x_i, y - y_i)$ with $y_i^2 = f(x_i)$. Moreover, the curves $H = 0$ and $y^2 = f$ intersect
in $2d$ points, counted with multiplicity, which implies that at most $2d$ prime ideals can occur in this factorization.

Since the ideal $(F - G \sqrt{f})$ is a product of some subset of these primes, there are at most $2^{2d}$ possibilities for it, and this ideal determines $F$ and $G$ up to a constant multiple.
Since one of $F$ or $G$ is required to be monic, $F$ and $G$ are therefore determined
in at most $2^{2d + 1}$ ways.
\end{proof}

We now bound the number of different polynomials $H$ yielding the same field $K$. This is a variation of \cite[Lemma 3.1]{EV}, incorporating 
an improvement that was suggested there. 

The restriction \eqref{eq:def_H} won't be used in this bound, so we consider the larger set of polynomials
\[
S(Y; S_d) := \{ f = x^d + c_1 x^{d - 1} + \cdots + c_d \in \Z\Big[\frac{1}{p_1}\Big][x] \ : \ |c_i| \leq (CY)^d \}
\]
whose denominators are bounded by those of $f(x)$, subject to the condition that $K := \mathbb{Q}[x]/(f(x))$ is a field with Galois group $S_d$, and where 
$C$ is a constant depending only on $f$ and $d$. By construction, this set contains all polynomials constructed in
\eqref{eq:def_H}. For each number field $K$ of degree $d$, we then define
\[
M_K(Y) := \#\{ f \in S(Y;S_d) : \mathbb{Q}[x]/(f(x)) \simeq K\}
\]
to be the multiplicity with which $K$ is so constructed.

\begin{proposition}\label{prop:shape} We have
\begin{equation}\label{eq:MK_bound}
M_K(Y) \ll \max(Y^{d}\mathrm{Disc}(K)^{-1/2}, Y^{d/2}).
\end{equation}
\end{proposition}
\begin{proof} Embed $\mathcal{O}_K \hookrightarrow \mathbb{R}^n$ in the usual way, and 
let $\lambda_0, \lambda_1,\dots,\lambda_{d-1}$ denote the successive minima of $\mathcal{O}_K$, corresponding to 
vectors $\alpha_0 = 1, \alpha_1, \dots, \alpha_{d - 1} \in \mathcal{O}_K$. Note that all roots $\alpha$
of polynomials counted by $S(Y; S_d)$ are bounded rational multiples of algebraic integers with $|\alpha| \ll Y$.

If $\lambda_{d - 1} \ll Y$, then an integral basis for $\calO_K$ fits inside a box of side length $O(Y)$,
so that $M_K(Y) \ll Y^d \mathrm{Disc}(K)^{-1/2}$. 
Otherwise, let $k < d - 1$ be the largest integer for which $\lambda_k \leq Y$.
Then
\begin{equation}\label{eq:gon}
M_K(Y) \ll \frac{Y^{k + 1}}{\lambda_1 \lambda_2 \cdots \lambda_k}
 \ll \frac{Y^{k + 1}}{\mathrm{Disc}(K)^{1/2}} \cdot \lambda_{k+1}\dots\lambda_{d-1},
 \end{equation}
since $\lambda_1 \dots \lambda_{d-1} \asymp \mathrm{Disc}(K)^{1/2}$.  If $k \leq \frac{d}{2} - 1 $, then $M_K(Y) \ll Y^{\frac{d}{2}}$ by the first bound above. Otherwise, by 
\cite[Theorem 3.1]{BSTTTZ}, we have $Y < \lambda_{d-1} \ll \mathrm{Disc}(K)^{1/d}$, so that
\begin{align*}
M_K(Y) 
	\ll & \ \frac{Y^{k + 1}}{\mathrm{Disc}(K)^{1/2}} \mathrm{Disc}(K)^{\frac{d -k - 1}{d}} \\
	= & \ \mathrm{Disc}(K)^{\frac{1}{2}} \big( Y /\mathrm{Disc}(K)^{\frac{1}{d}} \big)^{k + 1} \\
	\ll & \   \mathrm{Disc}(K)^{\frac{1}{2}} \big( Y/ \mathrm{Disc}(K)^{\frac{1}{d}} \big)^{\frac{d}{2}} \\
	= & \ Y^{\frac{d}{2}}.
\end{align*}
\end{proof}

Finally, we require bounds on the number of $S_d$-fields of bounded discriminant.

\begin{proposition}\label{prop:ev}\cite{schmidt, EV} We have
\[
\#\mathcal{F}_d(X) \ll X^{\alpha(d)},
\]
where we may take 
\begin{equation}\label{eq:nf_bound}
\alpha(d) =
\begin{cases}
	\frac{d + 2}{4} & \textnormal{for any $d \geq 3$, and} \\
	\frac{d}{4} - \frac{3}{4} + \frac{1}{2d}     & \textnormal{for any $d \geq 16052$}.
\end{cases}
\end{equation}

\end{proposition}
\begin{proof}
The first bound is due to Schmidt \cite{schmidt}. In \cite[(2.6)]{EV}, 
Ellenberg and Venkatesh prove for any $d$
that for any positive integers $r$ and $k$ satisfying
\begin{equation}\label{eqn:ev-constraint}
{r + k \choose r} > \frac{d}{2}
\end{equation}
one may take
\begin{equation}\label{eq:EV_details}
\alpha(d) = \frac{4k}{d - 2} \cdot {r + 4k \choose r }.
\end{equation}
One immediately checks that the choice $r = 2$, $k = \lceil \sqrt{d} - 1 \rceil$
satisfies \eqref{eqn:ev-constraint} for $d > 129^2 = 16641$ and that \eqref{eq:EV_details} is stronger than \eqref{eq:nf_bound}.  By computer one further checks that for $d \geq 16052$, there is some $k$
satisfying \eqref{eqn:ev-constraint} with $r=2$ for which \eqref{eq:EV_details} yields \eqref{eq:nf_bound}.
\end{proof}

The above bounds are far from sharp, but the second bound on $\alpha(d)$ in \eqref{eq:nf_bound} is enough
in our proof. We expect that improvements to  \cite[(2.6)]{EV}, and hence to the range $d \geq 16052$,
should be possible.

\subsection{Assembling the ingredients}
Write $N_{E,d}(X)$ for the number of degree $d$, $S_d$-number fields $K$ with $|\mathrm{Disc}(K)| < X$ that are cut out by a $\overline{\mathbb{Q}}$-point of $E$.

We put the preceding steps together as follows:
\begin{itemize}
\item
The number of choices for the $a_i$ and $b_j$ is $\asymp Y^c$, where for $d$ even we compute that
\[
c = \sum_{i = 1}^{d/2} i + \sum_{j = 2}^{d/2} \left( j - \frac{3}{2} \right)
= \frac{d^2}{4} - \frac{d}{4} + \frac12,
\]
and a similar computation with $d$ odd yields the same result.

\item
By Hilbert irreducibility (Theorem \ref{thm:hit}) and Lemma \ref{lem:bound_mult}, we therefore obtain $\asymp Y^c$ different $\alpha$
as roots of polynomials $H(x)$ which generate $S_d$ fields, and for which $(\alpha, \frac{F(\alpha)}{G(\alpha)})$ is a point on $E(\overline{\mathbb{Q}})$.  
Since these polynomials have bounded denominators, the discriminant of each of these polynomials, and thus of the fields themselves,
is $\ll Y^{d^2 - d}$. Write $$X := C_1 Y^{d^2 - d}$$ for a bound on these discriminants, where $C_1$ is a constant depending only
on $f$ and $d$.
\item
Following the strategy in (3.2) of \cite{EV}, by
Proposition \ref{prop:shape}  we therefore have
\begin{equation}\label{eq:EV_strategy}
\sum_{|\Disc(K)| \leq X} M_K(Y)
\gg Y^c,
\end{equation}
where the sum is over the fields $\Q(\alpha)$ generated by the $\alpha$
as described above, which is a subset of the fields counted by $N_{E,d}(X)$.
\end{itemize}

We are now ready to finish. We first use Propositions \ref{prop:shape} and \ref{prop:ev} to bound the contribution to \eqref{eq:EV_strategy} from fields of small discriminant.
With $\alpha(d) = \frac{d+2}{4}$ in \eqref{eq:nf_bound}, we have for $T \leq Y^d$ that
\begin{equation}\label{eq:ld_schmidt}
\sum_{\mathrm{Disc}(K) \leq T} M_K(Y) \ll \sum_{\mathrm{Disc}(K)\leq T} \frac{Y^{d}}{\mathrm{Disc}(K)^{1/2}} \ll Y^d T^{\frac{d+2}{4}-\frac{1}{2}} = Y^d T^{d/4},
\end{equation}
which is $o(Y^c)$ with the choice $T = Y^{d - 5 + \frac{2}{d} - \epsilon}$. We thus have from \eqref{eq:EV_strategy} that
\begin{equation}\label{eq:EV_strategy_2}
\sum_{T < |\Disc(K)| \leq X} M_K(Y)
\gg Y^c.
\end{equation}
By Proposition \ref{prop:shape}, $M_K(Y) \ll Y^d/T^{1/2}$ for each $K$ in the sum, and a bit of algebra shows that
\[
N_{E,d}(X) 
\gg Y^c \big( Y^d/T^{1/2} \big)^{-1} \gg X^{\gamma - \epsilon}
\]
with 
\begin{equation}\label{eq:gamma_schmidt}
\gamma 
	= \frac{c - d + \frac{1}{2}(d-5+2/d)}{d^2-d}
	= \frac{1}{4} - \frac{d^2+4d-2}{2d^2(d-1)}.
\end{equation}
This yields the stated value of $c_d$ in Theorem \ref{thm:large-degree} and Theorem \ref{thm:general}.

\medskip
If we instead assume the slightly better bound $\alpha(d) = \frac{d}{4} - \frac{3}{4} + \frac{1}{2d}$ as in the hypotheses of Theorem \ref{thm:field-improvement}, then
we find that the contribution from those fields $K$ with $\mathrm{Disc}(K) \leq T$ is $o(Y^c)$ for any $T \ll Y^{d-\epsilon}$. In \eqref{eq:EV_strategy_2} we now
have $M_K(Y) \ll Y^{d/2+\epsilon}$ for each $K$, yielding 
$N_d(X) \gg X^{\gamma - \epsilon}$ with
\begin{equation}\label{eq:gamma_linnik}
\gamma = \frac{c - \frac{d}{2}}{d^2 - d} 
=\frac{1}{4} - \frac{1}{2d}.
\end{equation}

Combined with Lemma \ref{lem:fractional-points}, this yields Theorem \ref{thm:large-degree} apart from the claim about the root number $w(E,\rho_K)$. To control 
the root number, we use Lemma \ref{lem:control_sign}.
By Corollary \ref{cor:sign}, there is some fixed power $M$ of the conductor $N_E$ such that if $F(x)$ and $G(x)$ lie in fixed congruence classes $\pmod{M}$, then the root number $w(E,\rho_K)$ depends only on the sign of the discriminant of $H(x)$.  Therefore, Lemma \ref{lem:control_sign} implies, for each $\varepsilon = \pm 1$, that a positive proportion of the fields $K$ counted by $N_{E,d}(X)$ have $w(E,\rho_K) = \varepsilon$.  This is Theorem \ref{thm:large-degree}.

\subsection{Limitations and conditional improvements}

Let $M_{E,K}(Y)$ be the multiplicity with which a given field $K$ arises from the construction \eqref{eq:def_H}. If we had the bound
$M_{E, K}(Y) \ll Y^\epsilon$, then this would yield Theorem \ref{thm:general} with 
\begin{equation}\label{eqn:cd-limit}
c_d = \frac{c}{d^2-d} = \frac{1}{4}+ \frac{1}{2(d^2-d)},
\end{equation}
the limitation of our method at present.  We do not know how to establish this bound on $M_{E,K}(Y)$ unconditionally, even on average over $K$, but we can show
that this follows from well known open conjectures.

One feature we have heretofore ignored is that the points $(\alpha, \frac{F(\alpha)}{H(\alpha)})$ are integral away from a single rational prime. 
We can bound the number of such points using a special case of a result of Helfgott and Venkatesh \cite[Theorem 3.8]{HelfgottVenkatesh}. 
Let $S$ be a finite set of places of $\mathbb{Q}$. Then, for each degree $d$ field $K$, the number of $K$-rational points on $E$
with canonical height at most $h$ and which are integral at all places not lying over $S$, is
\begin{equation}\label{eqn:helfvenk}
O_{S,f,d}\left((1+\log h)^2 e^{.28 \cdot \mathrm{rk}(E(K))}\right),
\end{equation}
where the implied constant depends on $d$, $S$, and the Weierstrass equation $E \colon y^2 = f(x)$.
In our case, the points $(\alpha, \frac{F(\alpha)}{G(\alpha)})$ have canonical height $\ll \log Y$.  Thus, \eqref{eqn:helfvenk} implies that 
$$
M_{E,K}(Y) \ll Y^{\epsilon + 0.28 \frac{\mathrm{rk}(E(K))}{\log Y}}.
$$  
We expect the rank of $E(K)$ to be $o(\log D_K) = o(\log Y)$ for every $K$, from which we would obtain $M_{E,K}(Y) \ll Y^\epsilon$.
This would yield Theorem \ref{thm:general} with $c_d$ as in \eqref{eqn:cd-limit}.  

\smallskip
Unfortunately, this pointwise bound on the rank appears to be out of reach of algebraic methods. However, we may deduce it from the conjectural bound
$\#\mathrm{Cl}(K(E[2]))[2] \ll \mathrm{Disc}(K(E[2]))^\epsilon$ for each $K\in \mathcal{F}_d(X)$.  In particular, the rank of $E(K)$ is bounded by that of the $2$-Selmer group $\mathrm{Sel}_2(E_K)$.
By a classical $2$-descent \cite[Proposition X.1.4]{Silverman} we have in turn that $|\mathrm{Sel}_2(E_K)| \ll |\mathrm{Cl}(K(E[2]))[2]|^2$.
The field $K(E[2])$ is at most a degree $6$ extension of $K$ and is unramified away from $2\Delta_E$,  so its discriminant satisfies $\mathrm{Disc}(K(E[2])) \ll \mathrm{Disc}(K)^6$. We therefore
have the chain of inequalities 
\[
\mathrm{rk}(E(K)) 
	\leq \mathrm{rk}(\mathrm{Sel}_2(E_K)) 
	\ll_{d,f} \log |\mathrm{Cl}(K(E[2]))[2]|
	\ll_! \epsilon \log (\mathrm{Disc}(K))
	\ll_d \epsilon \log Y,
\]
where only the inequality marked $\ll_!$ is conjectural.  Combined with \eqref{eqn:helfvenk}, this would yield $M_{E,K}(Y) \ll Y^\epsilon$, and thereby that \eqref{eqn:cd-limit} is admissible in Theorem \ref{thm:general}.


Alternatively, if we assume that the $L$-function $L(s,E,\rho_K)$ is entire, then the Birch and Swinnerton-Dyer conjecture provides an analytic way of accessing the rank of $E(K)$.  Unfortunately, here too we run into an obstacle, with unconditional methods only being able to show that the analytic rank is $O(\log N_E^{d-1}D_K^2) = O(\log Y)$.  However, if we are willing to assume that $L(s,E,\rho_K)$ satisfies the generalized Riemann hypothesis, then from \cite[Proposition 5.21]{IwaniecKowalski}, we obtain the slight improvement
\[
\mathrm{ord}_{s=1/2}L(s,E,\rho_K) \ll \frac{\log Y}{\log\log Y},
\]
which is sufficient to conclude that $M_{E,K}(Y) \ll Y^\epsilon$.

The outcome of this discussion is the following proposition.

\begin{proposition}
Let $E/\mathbb{Q}$ be an elliptic curve and let $K \in \mathcal{F}_d(X)$.  Suppose that either $L(s,E,\rho_K)$ is entire and satisfies both the Birch and Swinnerton-Dyer conjecture and the generalized Riemann hypothesis, or that $\#\mathrm{Cl}(K(E[2]))[2] \ll D_K^\epsilon$.  Then $M_{E,K}(Y) \ll Y^\epsilon$.

In particular, if either holds for all $K \in \mathcal{F}_d(X)$, then Theorem \ref{thm:general} holds with 
\[
c_d = \frac{1}{4}+ \frac{1}{2(d^2-d)}.
\]
\end{proposition}

\bibliographystyle{alpha}
\bibliography{nonabeliantwists}

\end{document}